\documentclass[11pt]{amsart}

\usepackage{ mathdots }

\usepackage[usenames,dvipsnames]{pstricks}
\usepackage{epsfig}
\usepackage{pst-grad} 
\usepackage{pst-plot} 
\usepackage{auto-pst-pdf}

\usepackage{psfrag}
\usepackage{amsmath}
\usepackage{amssymb}
\usepackage[english]{babel}
\usepackage{amsthm}
\usepackage{amsfonts}
\usepackage[applemac]{inputenc}
\usepackage{tikz}
\usetikzlibrary{arrows}
\usetikzlibrary{patterns}
\usetikzlibrary{shapes}
\usepackage{mathrsfs}
\usepackage{bbm}

\newtheorem{teo}{Theorem}
\newtheorem{lemma}[teo]{Lemma}

\newtheorem{coro}[teo]{Corollary}
\newtheorem{propo}[teo]{Proposition}

\newcommand{\R}{\mathbb{R}}
\newcommand{\ii}{\mathrm{i}}

\newcommand{\eps}{\epsilon}
\newcommand{\N}{\mathbb{N}}
\newcommand{\C}{\mathbb{C}}
\newcommand{\CP}{\mathbb{C}\textrm{P}}
\newcommand{\RP}{\mathbb{R}\mathrm{P}}

\newcommand{\Z}{\mathbb{Z}}

\newcommand{\sym}{\textrm{Sym}}

\newcommand{\Q}{\mathcal{Q}}

\theoremstyle{remark} 
\newtheorem{remark}[]{Remark}

\newtheorem{example}[]{Example}

\title[Complexity of intersections of real quadrics]{Complexity of intersections of real quadrics and topology of symmetric determinantal varieties}

\author{ A. Lerario}
\thanks{Department of Mathematics, Purdue University.}

\begin{document}

\maketitle

\begin{abstract}Let $X$ be the intersection in $\RP^n$ of $k$ quadrics, i.e. the zero locus of the homogeneous, degree two polynomials $q_1, \ldots, q_k$. Let also $W$ be the span of these polynomials in the space of all homogeneous degree two polynomials and for every $r\geq 0$ let $$\Sigma_W^{(r)}=\{q\in W\backslash\{ 0\}\,|\, \dim\ker (q)\geq r\}.$$ 
Notice that $\Sigma_W^{(1)}$ equals the (spherical) intersection of $W$ with the discriminant hypersurface in the space of quadratic polynomials; moreover for $r\geq 2$ and $W$ generic $\Sigma_W^{(r)}=\textrm{Sing}\big(\Sigma_W^{(r-1)}\big).$\\
We prove that for a generic choice of $q_1, \ldots, q_k$ the following formula holds for the total Betti number of $X$:
\begin{equation}\label{eqabstract}b(X)\leq b(\RP^n)+\frac{1}{2}\sum_{r\geq 1}b\big(\Sigma_W^{(r)}\big)\end{equation}
In the case we remove the nondegeneracy hypotesis the previous formula remains valid upon substitution of $\Sigma_W$ with a pertubation of it obtained by translating $W$ in the direction of a small negative definite quadratic form. The previous sum (both in the generic and the general case) contains at most $O(k)^{1/2}$ summands, as the sets $\Sigma_W^{(r)}$ turns out to be empty for $\binom{r+1}{2}\geq k.$ We study the topology of symmetric determinantal varieties, like the above $\Sigma_W^{(r)}$, and bound their homological complexities, with  particular interest at those defined on a sphere. Using formula (\ref{eqabstract}) and the results on the complexity of determinantal varieties, we prove the sharp bound:
$$b(X)\leq O(n)^{k-1}$$
thus improving Barvinok's style bounds (recall that the best previously known bound, due to Basu, has the shape $O(n)^{2k+2}$). 

\end{abstract}
 \section*{Acknowledgements} 
The author is grateful to S. Basu for many stimulating discussions and to A. A. Agrachev for his constant support. 
 \section{introduction}
 This paper addresses the question of bounding the topology of the set:
$$X=\textrm{intersection of $k$ quadrics in $\RP^n$}$$
More specifically we are interested in finding a bound for its \emph{homological complexity}  $b(X)$, namely the sum of its Betti numbers\footnote{From now on, unless differently specified, all homology and cohomology groups are assumed to be with coefficients in $\Z_2.$}.\\
The problem of bounding the topology of semialgebraic sets belonging to some specified family dates back to the work of J. Milnor, who proved that a semialgebraic set $S$ defined in $\R^n$ by $s$ polynomial inequalities of degree at most $d$ has complexity bounded by $b(S)\leq O(sd)^n$.
What is special about sets defined by \emph{quadratic inequalities} is that the two numbers $k$ and $n$ can be exchanged in the previous to give A. Barvinok's bound $b(S)\leq n^{O(s)}.$\\
This kind of \emph{duality} between the variables and the equations in the quadratic case is the leading theme of this paper, as we will discuss in a while.\\
Barvinok's bound, and its subsequent improvements made by M-F. Roy, S. Basu and D. Pasechnik in \cite{BaPaRo} and M. Kettner and S. Basu in \cite{BaKe}, concerns with sets defined by $s$ quadratic \emph{inequalities}. The most refined estimate for the complexity of such sets is polynomial in $n$ of degree $s$, but since we need two inequalities to produce an equality, this bound when applied to the set $X$ of our interest produces:
$$\textrm{Basu's bound:}\quad b(X)\leq O(n)^{2k+2}$$
In this paper we focus on the algebraic case rather than the semialgebraic one. From the viewpoint of classical algebraic geometry our problem can be stated as follows. We are given a linear system $W$ of real quadrics, i.e. the span of $k$ quadratic forms in the space of all homogeneous degree two polynomials, and we consider the base locus $X=X_W$ of $W$, i.e. the set of points in $\RP^n$ where all these forms vanish. What we consider to be the \emph{dual} object to the previous one is the set $\Sigma_W$ of critical points of $W$, i.e. the set consisting of those nonzero elements in $W$ that are degenerate. More generally we can consider the set $\Sigma_W^{(r)}$ of those nonzero quadratic forms in $W$ whose kernel has dimension at least $r$. Using this notation one of the main results of this paper is the following formula, which holds in the case $W$ is generic:
\begin{equation}\label{introboundgeneric}b(X_W)\leq b(\RP^n)+\frac{1}{2}\sum_{r\geq1} b\big(\Sigma_W^{(r)}\big).\end{equation}
The previous sum is finite, since for a generic $W$ the set $\Sigma_W^{(r)}$ is empty for $\binom{r+1}{2}\geq k$. In fact we notice that for every  natural $r$ and a generic $W$ the set $\Sigma_{W}^{(r+1)}$ coincides with the set of singular points of $\Sigma_W^{(r)}$; the codimension of this singular locus is exactly $\binom{r+1}{2}$, thus it is empty for $r>\frac{1}{2}(-1+\sqrt{8k-7}).$ If we remove the genericity assumption a similar formula can be proved, but a \emph{perturbation} of $\Sigma_W$ is introduced. More specifically we have to translate $W$ in the direction of a small negative definite quadratic form $-\eps q$, getting in this way an \emph{affine} space $W_\eps=W-\eps q$. We consider then the unit sphere in $W_\eps$ and the set $\Sigma_\eps^{(r)}$ of quadratic forms on this sphere where the kernel has dimension at least $r.$ The following formula holds (now for \emph{any} $X$):
\begin{equation}\label{introbound}b(X)\leq b(\RP^n)+\frac{1}{2}\sum_{r\geq1}b\big(\Sigma_\eps^{(r)}\big)\end{equation}
The same remark on codimensions  as above applies here and this sum is actually finite, containing no more than $O(k)^{1/2}$ summands. For a generic choice of $W$ these two constructions coincide, since the set $\Sigma_W$ deformation retracts on its intersection with any unit sphere in $W$. We will adopt this notation in the sequel and think at $\Sigma_W$ as at its homotopy equivalent intersection with a sphere.\\
Notice that from the point of view of \emph{complex} algebraic geometry the bound of (\ref{introboundgeneric}), even translated to the complex setting, does not say anything interesting since the generic case coincides with the complete intersection one, whose topology is determined. On the contrary in the real setting the topology of $X$, even a smooth one, can vary dramatically and our bound becomes more effective.
\vspace{0,3cm}
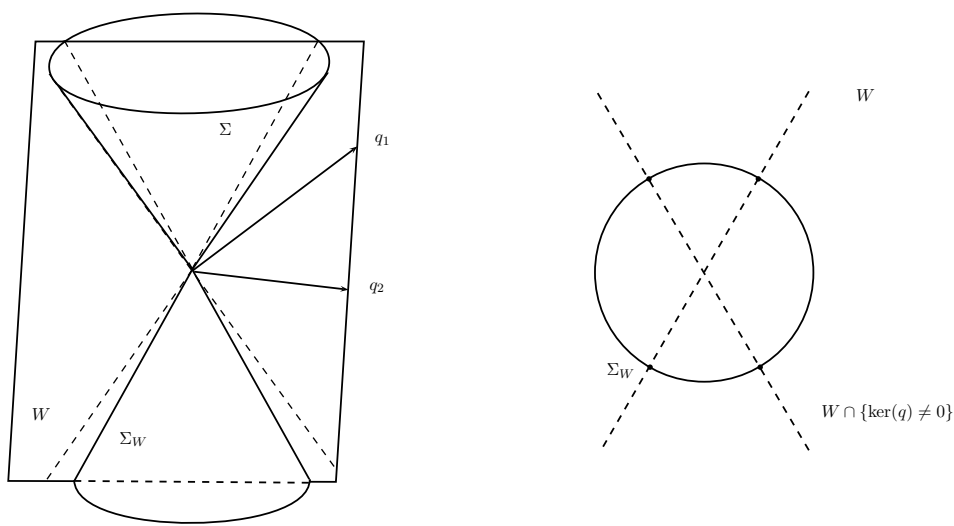
\begin{center}
\scalebox{0.6} 
{
\begin{pspicture}(0,-6.05)(21.391895,6.05)
\pspolygon[linewidth=0.04,fillstyle=solid](7.166867,-4.3596473)(1.0786866,-4.340118)(4.0778604,-0.0593187)
\pspolygon[linewidth=0.04,fillstyle=solid](0.0,-4.7502356)(7.26,-4.77)(7.8798966,4.9949408)(0.6033332,4.9949408)
\pspolygon[linewidth=0.04,fillstyle=solid](0.93242395,4.2528234)(7.067329,4.28489)(4.0778604,-0.0593187)
\psbezier[linewidth=0.04,fillstyle=solid](0.90721154,4.47054)(0.8072115,3.0053408)(7.1126237,3.0596006)(7.1126237,4.5448003)(7.1126237,6.03)(1.0072116,5.9357395)(0.90721154,4.47054)
\psline[linewidth=0.04cm](1.2066664,4.9949408)(6.8926244,4.9949408)
\psline[linewidth=0.03cm,linestyle=dashed,dash=0.16cm 0.16cm](1.243232,5.014471)(4.0770698,-0.04364715)
\psline[linewidth=0.03cm,linestyle=dashed,dash=0.16cm 0.16cm](6.8560586,4.9949408)(4.0770698,-0.04364715)
\psline[linewidth=0.04cm](4.0770698,-0.06317658)(6.691513,-4.7502356)
\psline[linewidth=0.03cm,linestyle=dashed,dash=0.16cm 0.16cm](4.0770698,-0.06317658)(0.8592927,-4.711177)
\psline[linewidth=0.03cm,linestyle=dashed,dash=0.16cm 0.16cm](0.98,4.17)(7.22,-4.4300003)
\psline[linewidth=0.04cm](4.0770698,-0.06317658)(1.4626259,-4.7502356)
\psbezier[linewidth=0.04,fillstyle=solid](1.46,-4.73)(1.5,-6.03)(6.68,-5.9700003)(6.68,-4.7183723)
\psline[linewidth=0.04cm,arrowsize=0.05291667cm 2.0,arrowlength=1.4,arrowinset=0.4]{->}(4.08,-0.09000007)(7.74,2.6699998)
\psline[linewidth=0.04cm,arrowsize=0.05291667cm 2.0,arrowlength=1.4,arrowinset=0.4]{->}(4.12,-0.11000008)(7.54,-0.51000005)
\usefont{T1}{ptm}{m}{n}
\rput(8.29291,2.8149998){$q_1$}
\usefont{T1}{ptm}{m}{n}
\rput(8.17291,-0.4650001){$q_2$}
\usefont{T1}{ptm}{m}{n}
\rput(0.73291016,-3.265){$W$}
\usefont{T1}{ptm}{m}{n}
\rput(2.77291,-3.825){$\Sigma_W$}
\usefont{T1}{ptm}{m}{n}
\rput(4.81291,3.0549998){$\Sigma$}
\pscircle[linewidth=0.04,dimen=outer](15.42,-0.1300001){2.44}
\psline[linewidth=0.04cm,linestyle=dashed,dash=0.16cm 0.16cm](13.06,3.85)(17.74,-4.07)
\psline[linewidth=0.04cm,linestyle=dashed,dash=0.16cm 0.16cm](17.74,3.8899999)(13.18,-3.99)
\psdots[dotsize=0.12](14.2,1.9499999)
\psdots[dotsize=0.12](16.62,1.9499999)
\psdots[dotsize=0.12](16.66,-2.23)
\psdots[dotsize=0.12](14.22,-2.23)
\usefont{T1}{ptm}{m}{n}
\rput(18.99291,3.8149998){$W$}
\usefont{T1}{ptm}{m}{n}
\rput(19.481455,-3.2450001){$W\cap\{\ker(q)\neq 0\}$}
\usefont{T1}{ptm}{m}{n}
\rput(13.57291,-2.305){$\Sigma_W$}
\psline[linewidth=0.04cm,linestyle=dashed,dash=0.16cm 0.16cm](1.46,-4.75)(6.66,-4.79)
\end{pspicture} 
}
\end{center}

\vspace{0.3cm}
The previous bounds (\ref{introboundgeneric}) and (\ref{introbound}) are just the mirror of the mentioned duality between the equations and the variables; the real tool encoding this duality is a spectral sequence introduced by A. A. Agrachev in \cite{Agrachev1} and developed by him and the author in \cite{AgLe}.\\
The powerful of this bounds is that they are intrinsic and different $X$ might produce different ones: for example a set $W$ whose nonzero forms have constant rank has base locus with complexity bounded by $b(\RP^n)$. On the other hand these bounds are sufficiently general to produce sharp numerical estimates. Indeed using them we can get the following, which improve Basu's one:
\begin{equation}\label{intronum}b(X)\leq O(n)^{k-1}\end{equation} 
This estimate is sharp in the following sense: if we let $B(k,n)$ be the maximum of the complexities over \emph{all} intersections of $k$ quadrics in $\RP^n$, then:
 \begin{equation}\label{B}B(k,n)=O(n)^{k-1}\end{equation}
 The upper bound for $B(k,n)$ is provided by (\ref{intronum}) and the lower one by the existence of a maximal real complete intersection, i.e. a complete intersection $M$ of $k$ real quadrics in $\CP^n$ satisfying $b(M_\R)=b(M).$ Such a complete intersection for $k\geq 2$ has the property:
 $$b(M_\R)=c_kn^{k-1}+O(n)^{k-2}$$ 
 (the smooth nonsingular quadric in $\CP^n$ has total Betti number equal to $n+\frac{1}{2}(1+(-1)^{n+1})).$
 The problem of determining the leading coefficient $c_k$ is far from trivial; we list as an example the first small values of $c_k$ starting from $k=2:$
 $$2, 1, \frac{2}{3}, \frac{1}{3}, \frac{2}{15}, \ldots$$
 For small values of $k$ the leading coefficient we get expanding in $n$ the r.h.s. of (\ref{introboundgeneric}) is the same as the complete intersection one. This provides:
 $$B(1,n)=n,\quad B(2,n)=2n,\quad B(3,n)=n^2+O(n)$$
 We conjecture that in general for $k\geq 2$ we have $B(k,n)=c_kn^{k-1}+O(n^{k-2}).$ The way this conjecture can be tackled, and indeed the way we produced the numerical bound (\ref{intronum}), is the study of the topology of symmetric determinantal varieties. In fact each set $\Sigma_\eps^{(r)}$ is defined by the vanishing of some minors of a symmetric matrix depending on parameters (in our case the parameter space is the unit sphere in $W_\eps$). The geometry of symmetric determinantal varieties over the complex numbers was studied in \cite{HarrisTu}, where the degrees of such varieties were explicitly computed. Here we do not need this degree computation; though we use the fact that determinantal varieties are defined by (possibly many) polynomials of small degree. Such a property, combined with Milnor's classical bound, produces  the general estimate:
 \begin{equation}\label{sigmabound}b\big(\Sigma_\eps^{(r)}\big)\leq (2n)^{k-1}+O(n)^{k-2}\end{equation}
Notice that if we plug this into (\ref{introboundgeneric}) we immediatley get $B(k,n)\leq O(n)^{k-1}$ (this follows at once using the fact that there are less than $O(k)^{1/2}$ terms in the sum we consider). On the other hand such algebraic sets, among those defined by polynomials of degree less than $n$, are very special. For example they have unavoidable singularities - that is the reason for the appearence of higher order terms in (\ref{introboundgeneric}). This is why we expect the leading coefficient of the bound (\ref{sigmabound}) not to be optimal. In fact for $k=1,2,3$ we bounded the complexities of these varieties with a direct argument, getting the optimal coefficient.
\\As an example for these ideas we compute the cohomology of the set $\Sigma$ of nonzero symmetric matrices with zero determinant. This set coincides with the discriminant hypersurface of homogeneous polynomials of degree two. The degree of this hypersurface is $n$ and Milnor's bound give $b(\Sigma)\leq O(n)^{\binom{n}{2}}.$ Though $\Sigma$ happens to be Spanier-Whitehead dual to a disjoint union of grassmannians and:
\begin{equation}\label{discriminant}H^*(\Sigma)\simeq \bigoplus_{j=0}^nH_*(\textrm{Gr}(j,n))\end{equation}
In particular the complexity of $\Sigma$ is exactly $2^n,$ a number which is much smaller than Milnor's prediction.\\
The paper is organized as follows. Section 2 gives an account of the known numerical bounds present in the literature. Section 3 introduces the spectral sequence approach, from which it is shown to be possible to recover Barvinok's bound. Section 4 deals with symmetric determinantal varieties and contains the proof of (\ref{discriminant}). Section 5 is the technical bulk of the paper and deals with the transversality arguments needed in order to prove (\ref{introbound}). Section 6 contains the proof of (\ref{introboundgeneric}) and (\ref{introbound}). Section 7 contains the proof of the numerical translation (\ref{intronum}) of the previous bounds as well as the discussion on its sharpness. Section 8 discusses some examples.\\
From now on all algebraic sets are assumed to be \emph{real} (in particular projective spaces and Grassmannians are the real ones) unless differently specified. Similarly all homology and cohomology groups are with coefficients in $\Z_2$.
 \section{Complexity of intersections of real quadrics}
The aim of this section is to review the numerical bounds that can be derived from the literature for the homological complexity of $X$.\\
The first result is due to J. Milnor, who proved that if $Y$ is an algebraic set defined by homogeneous polynomials of degree at most $d$ in $\RP^n$, then $b(Y)\leq n d(2d-1)^{n-1}$ (this is the content of Corollary 3 of \cite{Milnor}). If we apply this bound to the set $X$ we immediately get the following estimate
$$\textrm{Milnor's bound:}\quad b(X)\leq 2n 3^n$$
What is interesting about this bound is that it does not depend on the \emph{number} of equations defining $Y$ (respectively $X$), but only on their degrees.\\
On a different perspective A. Barvinok studied the complexity of basic semialgebraic subsets of $\R^n$ defined by a fixed number of inequalities of degree at most two. Using the main result from \cite{Ba} we can derive another bound, whose shape is different from the previous one\footnote{According to \cite{BPR} in this context the notation $f(l)=O(l)$ means that there exists a natural number $b$ such that the inequality $f(l)\leq bl$ holds for every $l\in \N$.}

$$\textrm{Barvinok's bound:}\quad b(X)\leq n^{O(k)}$$
\begin{proof}Theorem (1.1) of \cite{Ba} states the following: for every $k\in \mathbb{N}$ there exists a polynomial $P_k$ of degree $O(k)$ such that for any semialgebraic set $Y$ defined by $k$ inequalities of degree at most two in $\R^n$ we have $b(Y)\leq 				P_k(n)=n^{O(k)}.$\\
We decompose our $X$ into its affine part $A$ and its part at infinity $B$ and we use a Mayer-Vietoris argument. More specifically we let $A=X\cap \{x_0\neq 0\}$ and $B=X\cap \{x_0^2\leq \eps\}$. Now $A$ is defined by $2k$ quadratic inequalities in $\R^n$ (each equation is equivalent to a pair of inequalities) and $B$ by $k$ quadratic equations in $\RP^n$ (in fact this set for small $\epsilon$ deformation retracts to $X\cap\{x_0=0\}$). The intersection $A\cap B$ is defined in $\R^n$ by $2k+1$ quadratic inequalities: those defining $A$ plus the one defining a big ball. We apply now Theorem (1.1) of \cite{Ba} to $A$ and $A\cap B$ to get a bound of the form $n^{O(k)}$ for their total Betti numbers. Induction on $n$ and the Mayer-Vietoris long exact sequence of the semialgebraic pair $(A,B)$ finally give:
$$b(X)\leq b(A)+b(B)+b(A\cap B)\leq n^{O(k)}$$ \end{proof}
The subtlety about the previous bound is the implied constant in its definition: indeed in Theorem (1.1) of \cite{Ba} this implied constant is at least two. This provides an implied constant of at least \emph{four} in Barvinok's bound. The work \cite{BaKe} of S. Basu and M. Kettner provides a better estimate for this constant:
$$\textrm{Basu's bound}:\quad b(X)\leq O(n)^{2k+2}$$
\begin{proof}Corollary 1.7 of \cite{BaKe} states the following: let $S$ be a semialgebraic subset of $\R^n$ defined by $k$ quadratic inequalities. Then 
		$$b(S)\leq \frac{n}{2}\sum_{j=0}^k \binom{k}{j} \binom{n+1}{j}2^j=s(k,n)$$
Let us show first that $s(k,n)$ behaves asymptotically as $O(n)^{k+1}$; indeed let us prove that $\lim_n \frac{\log s(k,n)}{\log n}=k+1$.
 Notice first that for every $k$ there exists $C_k>0$ such that for every $n$ 
		$$\binom{n+1}{k}\leq \sum_{j=0}^k\binom{k}{j} \binom{n+1}{j}2^j\leq C_k \binom{n+1}{k}$$
		The existence of such a $C_k$ is due to the fact the the number of terms we are adding and the number $\binom{k}{j}$ do not depend on $n$ but only on $k$.
		Using Stirling's asymptotic at infinity $n!\sim \sqrt{2\pi n}(\frac{n}{e})^n$ we can write:
		$$\binom{n+1}{k}\sim \frac{1}{k! e^{k}}\frac{(n+1)^{n+1}}{(n+1-k)^{n+1-k}}\sqrt{\frac{n+1}{n+1-k}}\sim A_k n^{k}$$
		for some constant $A_k>0$. The inequalities $\frac{n}{2}\binom{n+1}{k}\leq s(k,n)\leq \frac{C_k n}{2}\binom{n+1}{k}$ and the previous asymptotic immediately give the limit.\\
		Proceeding now as in the proof of Barvinok's estimate, i.e. decomposing $X$ into its affine and infinity part and using Mayer-Vietoris bounds, provides the result.
	\end{proof}
The bound provided in \cite{BaKe} is the best known for semialgebraic sets defined by quadratic \emph{inequalities}. Surprisingly enough it turns out that in the special case of our interest, i.e. algebraic sets, the exponent of Basu's bound can be lowered down to $k-1$. This will be a straightforward consequence of a deeper approach for bounding the topology of $X$ with the complexity of some determinantal varieties associated (and in a certain sense \emph{dual}) to it. This is based on a spectral sequence argument and has strong consequences, besides the framework of bounding the topology of $X$.

\section{The spectral sequence approach}
In this section we will discuss a different approach for the study of the intersection of real quadrics. This was first introduced by A. A. Agrachev in \cite{Agrachev1} and \cite{Agrachev2} for the nonsingular case and then extended in \cite{AgLe} for the general case.\\
Let $\Q_{n}$ denote the vector space of homogeneous polynomials of degree two in $n$ variables, i.e. the space of quadratic forms over $\R^n$. Then $X$ is the zero locus in in the projective space of the elements $q_1, \ldots, q_k \in\Q_{n+1}$ and we consider the \emph{linear system}\footnote{In classical algebraic geometry $X$ is referred to as the base locus of $W$.} defined by these elements:
$$W=\textrm{span}\{q_1, \ldots, q_k\} \subset \Q_{n+1}$$
For a given quadratic form $p\in \Q_{n+1}$ we denote by $\ii^{+}(p)$ its positive inertia index, namely the maximal dimension of a subspace of $\R^{n+1}$ such that the restriction of $p$ to it is positive definite. The idea of the spectral sequence approach is to replace the geometry of $X$ with the one of the restriction of the function $\ii^+$ to $W$. More precisely let us consider the sets:
$$W^j=\{q\in W \,|\, \ii^+ (q)\geq j\},\quad j\geq1$$ 
Notice that none of these sets contains the zero and all of them are invariant under multiplication by a positive number, hence they deformation retracts to their intersections with a unit sphere in $W$ (with respect to any scalar product):
$$\Omega^j=W^j \cap \{\textrm{any unit sphere in $W$}\},\quad j\geq1$$
Even if not canonical (it depends on the choice of a scalar product on $W$) sometimes it is more convenient to use this family rather than the previous one; notice though that different scalar products will produce homeomorphic families. The spirit of this approach is to exploit the relation between $X$ and the filtration
 $$\Omega^{n+1}\subseteq \Omega^n \subseteq \cdots\subseteq \Omega^2\subseteq \Omega^1$$
 This is made precise by a Leray's spectral sequence argument, and produces the following
\begin{teo}$$\emph{\textrm{Agrachev's bound}}:\quad b(X)\leq n+1+\sum_{j\geq0}b(\Omega^{j+1})$$
\end{teo}
\begin{proof}Let $S_W$ be any unit sphere in $W$ and consider the topological space $B=\{(\omega,[x])\in S_W\times \RP^{n}\,|\,  (\omega q)(x)>0\}$ together with its two projections $p_{1}:B\to S_W$ and $p_{2}:B\to \RP^{n}.$ The image of $p_{2}$ is easily seen to be $\RP^{n}\backslash X$ and the fibers of this map are contractible sets, hence $p_{2}$ gives a homotopy equivalence $B\sim\RP^{n}\backslash X.$ Consider now the projection $p_{1};$ for a point $\omega \in S_W$ the fiber $p_{1}^{-1}(\omega)$ has the homotopy type of a projective space of dimension $\ii^{+}(\omega q)-1$, thus the Leray spectral sequence for $p_{1}$ converges to $H^{*}(\RP^{n}\backslash X)$ and has the second term $E_{2}^{i,j}$ isomorphic to $H^{i}(\Omega^{j+1})$. A detailed proof of the previous statements can be found in \cite{AgLe}. Since $\textrm{rk}(E_{\infty})\leq \textrm{rk}(E_{2})$ then $b(\RP^{n}\backslash X)\leq \sum_{j\geq0} b(\Omega^{j+1})$. Recalling that by Alexander-Pontryagin duality $H_{n-*}(X)\simeq H^{*}(\RP^{n},\RP^{n}\backslash X),$ then the exactness of the long cohomology exact sequence of the pair $(\RP^{n},\RP^{n}\backslash X)$ gives the desired inequality. \end{proof}
It interesting to notice that Agrachev's bound implies Barvinok's one. In fact let us fix a scalar product on $\R^{n+1}$; then the rule $\{\forall x\in \R^{n+1}\,:\,\langle x, Q x\rangle =q(x)\}$ defines a symmetric matrix $ Q$ whose number of positive eigenvalues equals $\ii^{+}(q).$ Consider the polynomial 
$$f(t,Q)=\det (Q-tI)=a_{0}(Q)+\cdots+a_{n}(Q)t^{n}\pm t^{n+1}$$ defined over $\R \times W=\R\times \textrm{span}\{Q_1, \ldots, Q_k\}$. Then by Descartes' rule of signs  the positive inertia index of $Q$ is given by the sign variation in the sequence $(a_{0}(Q), \ldots, a_{n}(Q)).$ Thus the sets $\Omega^{j+1}$ are defined on the unit sphere in $W$ by sign conditions (quantifier-free formulas)  whose atoms belong to a set of $n+1$ polynomials in $k$ variables and of degree less than $n+1$. For such sets we have the estimate, proved in \cite{BPR}: $b(\Omega^{j+1})\leq  n^{O(k)}$. Putting all this together we get:
	$$b(X)\leq n+1+\sum_{j\geq 0}b(\Omega^{j+1})\leq n^{O(k)}$$
\begin{example}\label{toy}Before going on we give an idea of the direction this spectral sequence approach will lead us; this will be just a motivation for next sections, the detailed theory being developed in the final part of the paper.
Let $S_W$ be the unit sphere in $W$ and let us assume that the set $$\Sigma_W=\{q\in S_W \,|\, \ker (q)\neq 0\}$$ is a \emph{smooth} manifold and each time we cross it the index function changes exactly by $\pm 1$. Then the components of this manifold are exactly the boundaries of the sets $\Omega^j$ and $b(\Sigma_W)=\sum_{\Omega^j \neq S_W}b(\partial \Omega^{j})$. On the other hand each $\partial \Omega^j$ is a submanifold of the sphere and it is not difficult to show that $b(\partial \Omega^j)=2 b(\Omega^j)$ (we will give an argument for this in Lemma \ref{boundary}). Combining all this together into Agrachev's bound we get:
\begin{equation}\label{toyeq}b(X)\leq n+1+\frac{1}{2}b(\Sigma_W)\end{equation}
which relates the topology of $X$ (the base locus of $W$) to the topology of $\Sigma_W$ (the singular locus of $W$). In the general case $\Sigma_W$ will not be smooth, nor the index function well behaving and a more refined approach is needed. This approach is based on the study of the topology of $\Sigma_W$ and its singularities. These two objects are very particular algebraic sets: they are defined by the set of points where a family of matrices has some rank degeneracy, i.e. they are determinantal varieties.
\end{example}

\section{Symmetric determinantal varieties}
The  aim of this section is to bound the topology and describe some geometry of symmetric determinantal varieties. In a broad sense these will be defined by rank degeneracy conditions of (algebraic) families of symmetric matrices. Recall that our interest comes from families of quadratic forms; the way to switch to symmetric matrices is simply by establishing a linear isomorphism between $\Q_n$ and $\sym_{n}(\R)$. This can be done, once a scalar product on $\R^n$ has been fixed, by associating to each quadratic form $q$ the matrix $Q$ defined by:
$$q(x)=\langle x, Qx \rangle \quad\forall x\in \R^n$$
Notice that the dimension of the vector space $\sym_{n}(\R)$ is $\binom{n}{2}.$\\
Suppose now that $Y$ is an algebraic subset of $\sym_n(\R)$; for every natural number $r$ we define its rank degeneracy locus
$$Y^{(r)}=\{Q\in Y\, |\,\textrm{dim} \ker(Q)\geq r\}$$
Using the bound of \cite{Milnor} we can immediately prove the following proposition, which exploits the idea that symmetric determinantal varieties have relatively simple topology.
\begin{propo}\label{bounddet}Let $Y$ be defined by polynomials of degree less than $d$ in $\sym_n(\R)$ and $\R^k$ be a subspace; let also $\delta=\max \{d, n-r+1\}$. Then:
$$b\big(Y^{(r)}\cap \R^k\big)\leq \delta (2\delta-1)^{k-1}$$
\begin{proof} The set $Y^{(r)}$ is defined in $\sym_n(\R)$ be the same equations defining $Y$ plus all the equations for the vanishing of minors of order $r+1$; these last equations have degree $r+1$. Once we intersect $Y^{(r)}$ with a linear space of dimension $k$, we get a set defined by equations of degree at most $\delta$ in $k$ variables and Milnor's estimate applies.
\end{proof}
\end{propo}
Let us fix now a scalar product also on the space $\sym_n(\R)$, e.g. we can take $\langle A, B\rangle =\frac{1}{2} \textrm{tr}(AB)$. We consider the set of singular matrices of norm one:
$$\Sigma=\{\|Q\|^2=1 \,,\, \det (Q)=0\}$$
This set is a deformation retract of the set of nonzero matrices with determinant zero and for $n>1$ is defined by equations of degree at most $n$ in $\sym_n(\R)$. The previous proposition would produce a bound of the form $an^{\binom{n}{2}}$ for its topological complexity; indeed the bound is much better, as shown in the next theorem.
\begin{teo}
$$H^*(\Sigma)\simeq \bigoplus_{j=0}^n H^*(\emph{\textrm{Gr}}(j,n))\quad \textrm{and} \quad b(\Sigma)=2^n$$
\begin{proof}In the space of all symmetric matrices let us consider the open set $A$ where the determinant does not vanish; this set deformation retracts to $S^N\backslash \Sigma$ and by Alexander-Pontryagin duality it follows that:
\begin{equation}\label{sigma}H^*(\Sigma)\simeq H_*(A)\end{equation}
On the other hand $A$ is the \emph{disjoint union} of the open sets:
$$G_{j,n}=\{\det(Q)\neq 0, \, \ii^+(Q)=j\},\quad j=0,\ldots,n$$
We prove that each of these sets is homotopy equivalent to a Grasmmannian; this, together with equation (\ref{sigma}), will give the desired result. More specifically we show that the semialgebraic map
$$p_k:G_{j,n}\to \textrm{Gr}(j,n)$$
which sends each matrix $Q$ to its positive eigenspace, is a homotopy equivalence. In fact let $\{e_1,\ldots, e_n\}$ be the standard basis of $\R^n$ and $E_j$ be the span of the first $j$ basis elements. The preimage of $E_j$ under $p_j$ equals the set of all symmetric block matrices of the form $$Q=\left( \begin{matrix} D^2& 0\\0&Q'\end{matrix}\right)$$
with $D$ diagonal invertible and $Q'$ invertible and negative definite, i.e. $Q'\in G_{0, n-j}$. In particular, since the set $G_{0, n-j}$ is an open cone, it is contractible and:
$$p_j^{-1}(E_j)\simeq (\R^+)^j\times G_{0, n-j}\quad\textrm{is contractible}$$
For $W\in \textrm{Gr}(j,n)$ let $M$ be any orthogonal matrix such that $MW=E_j$; then clearly $p_j^{-1}(W)=M^{-1}p_j^{-1}(E_j) M$ and all the fibers of $p_j$ are homeomorphic. Thus $p_j$ is a semialgebraic map with contractible fibers, hence a homotopy equivalence. The last part of the theorem follows from the well known fact that $b(\textrm{Gr}(j,n))=\binom{n}{j}$ and the formula $\sum_{j=0}^n \binom{n}{j}=2^n$
\end{proof}\end{teo}

Let $Z$ be the algebraic set of all singular matrices in $\sym_n(\R)$; we will be interested in greater generality in the filtration:
\begin{equation}\label{strataZ}\{0\}=Z^{(n)}\subset Z^{(n-1)}\subset \cdots \subset Z^{(2)}\subset Z^{(1)}=Z\end{equation}
We recall that each $Z^{(r)}$ is a real algebraic subset of $\sym_n(\R)$ of codimension 
$\binom{r+1}{2}$ and that the singular loci of these varieties are related by:
\begin{equation}\label{sing}\textrm{Sing}( Z^{(j)})=Z^{(j+1)}\end{equation}
References for this statements are \cite{Agrachev1} and \cite{AgLe}; in particular Proposition 9 of \cite{AgLe} shows that $Z$ is Nash stratified by the smooth semialgebraic sets $N_r=Z^{(r)}\backslash Z^{(r+1)}$ (see \cite{BCR} for the definition and properties of Nash stratifications).
Notice also that using the above notation we have the equalities $Y^{(r)}=Y\cap Z^{(r)}$ and $\Sigma=\{\|Q\|^2=1\}\cap Z^{(1)}$.\\
The degrees of the complexifications $Z_\C^{(r)}$ of these varieties are computed in \cite{HarrisTu}:
$$\deg Z_\C^{(r)}=\prod_{\alpha=0}^{r-1}\frac{\binom{n+\alpha}{r-\alpha}}{\binom{2\alpha+1}{\alpha}}=O(n)^{\frac{r(r+1)}{2}}  $$
Notice that they (and their hyperplane sections) have big degree but small topological complexity; this is due to the fact that can be defined by (many) equations of low degree.\\
Let us denote by $\ii^-(Q)$ the number of \emph{negative} eigenvalues of a symmetric matrix $Q$ and recall that
$$P_j=\{Q\in \sym_n(\R)\,|\, \ii^-(Q)\leq j\}, \quad0\leq j\leq n-1$$
is a (noncompact) topological submanifold of $\sym_n(\R)\simeq \Q_n$ with boundary (see \cite{Agrachev1}); let us set $$A_j=\partial P_j,\quad 0\leq j\leq n-1.$$ 
Here we show the picture for $ \sym_{2}(\R)\simeq \R^3.$ The set $P_0$ is the closed cone of positive semidefinite matrices, its boundary is a topological manifold; $P_1$ is the set of sign-indefinite matrices, it contains $P_0$ ant its boundary is again a topological manifold. The union $\partial P_0\cup \partial P_1$ is the set of singular matrices and this set is not a manifold: its singular locus (the zero matrix) is given by $\partial P_0\cap \partial P_1.$
\vspace{0.3cm}
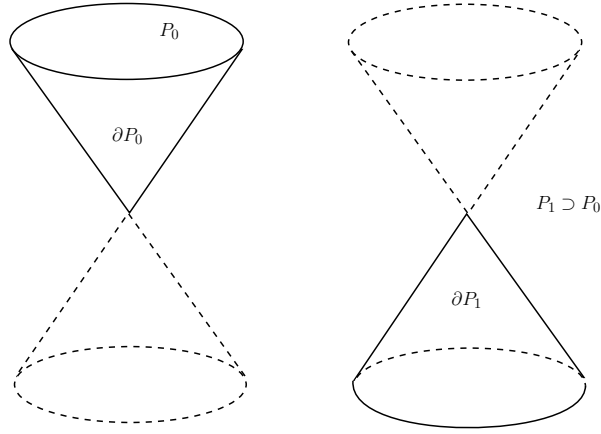
\begin{center}
\fontsize{15}{0}

\scalebox{0.5} 
{
\begin{pspicture}(0,-5.89)(16.201895,5.87)
\pspolygon[linewidth=0.04,fillstyle=solid](0.08,4.63)(6.187329,4.6248903)(3.1978602,0.2806814)
\psellipse[linewidth=0.04,dimen=outer,fillstyle=solid](3.12,4.84)(3.12,1.03)
\pspolygon[linewidth=0.04,fillstyle=solid](15.28,-4.0993185)(9.16,-4.19)(12.16214,0.25)
\usefont{T1}{ptm}{m}{n}
\rput(4.261455,5.115){$P_0$}
\usefont{T1}{ptm}{m}{n}
\rput(3.1414552,2.315){$\partial P_0$}
\usefont{T1}{ptm}{m}{n}
\rput(14.871455,0.515){$P_1\supset P_0$}
\usefont{T1}{ptm}{m}{n}
\rput(12.161455,-2.045){$\partial P_1$}
\pspolygon[linewidth=0.04,linestyle=dashed,dash=0.16cm 0.16cm,fillstyle=solid](6.28,-4.0793185)(0.17267115,-4.074209)(3.16214,0.27)
\rput{-180.0}(6.44,-8.578637){\psellipse[linewidth=0.04,linestyle=dashed,dash=0.16cm 0.16cm,dimen=outer,fillstyle=solid](3.22,-4.2893186)(3.1,1.03)}
\pspolygon[linewidth=0.04,linestyle=dashed,dash=0.16cm 0.16cm,fillstyle=solid](9.08,4.61)(15.187329,4.6048903)(12.19786,0.2606814)
\psellipse[linewidth=0.04,linestyle=dashed,dash=0.16cm 0.16cm,dimen=outer,fillstyle=solid](12.12,4.82)(3.12,1.03)
\rput{-180.0}(24.48,-8.658637){\psellipse[linewidth=0.04,linestyle=dashed,dash=0.16cm 0.16cm,dimen=outer,fillstyle=solid](12.24,-4.3293185)(3.1,1.03)}
\psbezier[linewidth=0.04,fillstyle=solid](9.14,-4.19)(9.0,-5.87)(15.32,-5.790941)(15.32,-4.319339)
\end{pspicture} 
}
\end{center}
\vspace{0.3cm}
The following proposition describes in more detail the structure of the sets $Z^{(r)}$, using the combinatorics of the $A_j$'s.
\begin{propo}\label{zetar}For every $r\geq 0$ let $I_r$ be the set of all the subsets $\alpha$ of $\{0, \ldots, n-1\}$ consisting of $r$ consecutive integers. Then
$$Z^{(r)}=\bigcup_{\alpha \in I_r} \bigcap_{j\in \alpha} A_j.$$

\end{propo}
\begin{proof}
For $l\geq 0$ let us say that a matrix $Q$ has the property $s(l)$ if there exists a sequence $\{Q_n\}_{n\geq0}$ converging to $Q$ such that $\ii^-(Q_n)\geq l.$ Using this notation we have that $A_j=\{\ii^-(Q)\leq j\textrm{ and $Q$ has the property $s(j+1)$}\}$.
From this it follows that for every $r\geq 0$
$$A_i\cap A_{i+r-1}=\{\ii^-(Q)\leq i\textrm{ and $Q$ has the property $s(i+r)$}\}$$
which also says that $A_i\cap A_{i+1}\cap \cdots \cap A_{i+r-1}=A_i\cap A_{i+r-1}$. Let now $Q\in Z^{(r)}$ and $M$ be an orthogonal matrix such that:
$$M^{-1}Q M=\textrm{diag}(-\lambda_1^2,\ldots, -\lambda_{\ii^-(Q)}^2, \mu_{\ii^-(Q)+1}, \ldots, \mu_{n-r}, 0, \dots, 0)$$
with the $\lambda_i$'s greater than zero. Let now $D_n$ be defined by changing each zero on the diagonal of the previous matrix to $-\frac{1}{n}.$
Then if we set $Q_n=M D_n M^{-1}$ we have that $Q$ satisfies property $s(\ii^-(Q)+r)$ and thus it belongs to $\bigcup_{\alpha \in I_r} \bigcap_{j\in \alpha} A_j.$ Viceversa let $Q$ be in $\bigcup_{\alpha \in I_r} \bigcap_{j\in \alpha} A_j$; then $Q$ satisfies $s(\ii^-(Q)+r)$  and there exist $\{Q_n\}_{n\geq 0}$ such that:
$$\dim \ker Q=n-\ii^+(Q)-\ii^-(Q)\geq n-\ii^+(Q_n)-\ii^-(Q_n)+r\geq r$$
(for the inequality $\ii^+(Q_n)\geq \ii^+(Q)$ we have used that $\{\ii^+(Q)\geq j\}$ is an open set), i.e. $Q$ is in $Z^{(r)}.$

\end{proof}

\section{Transversality arguments}
In this section we will discuss the following idea. Suppose we are given $X$ by the vanishing of some quadratic polynomials in $\RP^n$ and let $W$ be the span of these polynomials as in the previous sections. The homological complexity of $X$ (up to a $n+1$ term) can be bounded, using Agrachev's bound, with the sum of the complexities of the sets $\Omega^j$. To have an alternate description of these sets let us introduce the following notation. Let $q_1,\ldots, q_k\in \Q_{n+1}$ be the quadratic forms defining $X$ and $q: S^{k-1} \to \Q_{n+1}$ the map defined by:
$$\omega=(\omega_1, \ldots, \omega_k) \mapsto \omega q= \omega_1 q_1+\ldots+\omega_k q_k$$
The map $q$ is the restriction to the sphere of the linear map sending the standard basis of $\R^k$ to $\{q_1, \ldots, q_k\}.$ We redefine now:
$$\Omega^j=\{\omega \in S^{k-1}\, | \, \ii^+(\omega q)\leq j\},\quad j\geq 1$$ If $q_1,\ldots, q_k$ are linearly independent, then this definition agrees with previous one; if they are not linearly independent the map $q$ is no longer an embedding, though a look at the proof of Agrachev's bound shows that it still holds:
$$b(X)\leq n+1+\sum_{j\geq 0}b(\Omega^{j+1})$$
(it is sufficient to use the set $B'=\{(\omega, [x])\in S^{k-1}\times \RP^n\, |\, (\omega q)(x)\geq 0\}$ instead of $B$ and the proof works the same; actually it can also be proved that these new sets deformation retracts to the previously defined ones). The question we address is now the following: what happens if we perturb the map $q$?\\
The perturbations we will be interested in are those of the form:
$$q_\eps:\omega\mapsto \omega q-\eps p$$
where $p$ is a positive definite quadratic form; in other words we will be interested in small \emph{affine} translations $q-\eps p$ of the map $q$.
It turns out that if $p$ is a positive definite quadratic form and $\eps>0$ is small enough then each set $\Omega^j$ is homotopy equivalent to the set:
$$\Omega_{n-j}(\eps)=\{\omega \in S^{k-1}\, |\, \ii^{-}(\omega q-\eps p)\leq n-j\}$$
where $\ii^{-}$ denotes the negative inertia index, i.e. $\ii^{-}(\omega q-\eps p)=\ii^{+}(\eps p-\omega q).$ In particular the Betti numbers of $\Omega^{j+1}$ and of its perturbation $\Omega_{n-j}(\eps)$ are the same, as proved in the following lemma from \cite{Le3}.
\begin{lemma}
\label{union}For every positive definite form $p\in \Q_{n+1}$  and for every $\eps>0$ sufficiently small
$$b(\Omega^{j+1})=b(\Omega_{n-j}(\eps)).$$
\begin{proof}
Let us first prove that $\Omega^{j+1}=\bigcup_{\eps>0}\Omega_{n-j}(\eps).$\\
Let $\omega \in \bigcup_{\eps>0}\Omega_{n-j}(\eps);$ then there exists $\overline{\eps}$ such that $\omega\in \Omega_{n-j}(\eps)$ for every $\eps<\overline{\eps}.$ Since for $\eps$ small enough $$\ii^{-}(\omega q-\eps p)=\ii^{-}(\omega q)+\dim (\ker (\omega q))$$ then it follows that $$\ii^{+}(\omega q)=n+1-\ii^{-}(\omega q)-\dim (\ker \omega q)\geq j+1.$$ Viceversa if $\omega \in \Omega^{j+1}$ the previous inequality proves $\omega\in \Omega_{n-j}(\eps)$ for $\eps$ small enough, i.e. $\omega \in \bigcup_{\eps>0}\Omega_{n-j}(\eps).$\\
Notice now that if $\omega\in \Omega_{n-j}(\eps)$ then, eventually choosing a smaller $\eps$, we may assume $\eps$ properly separates the spectrum of $\omega$ and thus, by continuity of the map $ q$, there exists an open neighborhood of $\omega$, $U$, such that $\eps$ properly separates also the spectrum of $\eta q$ for every $\eta \in U$. Hence every $\eta \in U$ also belongs to $\Omega_{n-j}(\eps)$. From this consideration it easily follows that each compact set in $\Omega^{j+1}$ is contained in some $\Omega_{n-j}(\eps)$ and thus $$\varinjlim_{\eps}\{H_{*}(\Omega_{n-j}(\eps))\}=H_{*}(\Omega^{j+1}).$$ It remains to prove that the topology of $\Omega_{n-j}(\eps)$ is definitely stable in $\eps$ going to zero. Consider the semialgebraic compact set $S_{n-j}=\{(\omega, \eps)\in S^{k-1}\times [0, \infty)\, |\, \ii^{-}(\omega q-\eps p)\leq n-j\}$. By Hardt's triviality theorem (see \cite{BCR}) we have that the projection $(\omega, \eps)\mapsto \omega$ is a locally trivial fibration over $(0,\eps)$ for $\eps$ small enough; from this the conclusion follows.
\end{proof}
\end{lemma}

The following is a variation of Lemma 4 of \cite{Le3} and describes the structure of the sets of degenerate quadratic forms on the 'perturbed sphere'.
We recall that the space $Z$ of all degenerate forms in $\Q_{n+1}$ admits the semialgebraic Nash stratification $Z=\coprod N_r $ where $N_r=Z^{(r)}\backslash Z^{(r+1)}$ (as above we use the linear identification between quadratic forms and symmetric matrices). 
\begin{lemma}\label{perturb}There exists a positive definite form  $p\in \Q_{n+1}$ such that for every $\eps>0$ small enough the map $q_\eps:S^{k-1}\to \Q_{n+1}$ defined by:
$$\omega \mapsto \omega q-\eps p$$ 
is transversal to all strata of $Z=\coprod N_r.$ In particular $q_\eps^{-1}(Z)=\coprod q_\eps^{-1}(N_r)$ is a Nash stratification, the closure of $q_\eps^{-1}(N_r)$ equals $q_\eps^{-1}(Z^{(r)})$ and 
$$\emph{\textrm{Sing}}\big(q_\eps^{-1}\big(Z^{(r)}\big)\big)= q_\eps^{-1}\big(Z^{(r+1)}\big)$$
\end{lemma}
\begin{proof}Let $\Q^{+}$ be set of positive definite quadratic forms in $\Q_{n+1}$ and consider the map $F:S^{k-1}\times \Q^{+}$ defined by 
$$(\omega, p)\mapsto \omega q-p.$$
Since $\Q^{+}$ is open in $\Q,$ then $F$ is a submersion and $F^{-1}(Z)$ is Nash-stratified by $\coprod F^{-1}(N_{i}).$ Then for $p\in \Q^{+}$ the evaluation map $\omega \mapsto f(\omega)-p$ is transversal to all strata of $Z$ if and only if $p$ is a regular value for the restriction of the second factor projection $\pi:S^{k-1}\times \Q^{+}\to \Q^{+}$ to each stratum of $F^{-1}(Z)=\coprod F^{-1}(N_{i}).$
Thus let $\pi_{i}=\pi|_{F^{-1}(N_{i})}:F^{-1}(N_{i})\to \Q^{+};$ since all datas are smooth semialgebraic, then by semialgebraic Sard's Lemma (see \cite{BCR}), the set $\Sigma_{i}=\{\hat{q}\in \Q^{+}\, | \, \hat{q}\textrm{ is a critical value of $\pi_{i}$}\}$ is a semialgebraic subset of $\Q^{+}$ of dimension strictly less than $\dim (\Q^{+}).$ Hence $\Sigma=\cup_{i}\Sigma_{i}$ also is a semialgebraic subset of $\Q^{+}$ of dimension $\dim (\Sigma)<\dim (\Q^{+})$ and for every $p\in \Q^{+}\backslash \Sigma$ the map $\omega\mapsto f(\omega)-p$ is transversal to each $N_{i}.$ Since $\Sigma$ is semialgebraic of codimension at least one, then there exists $p\in \Q^{+}\backslash \Sigma$ such that $\{t p\}_{t>0}$ intersects $\Sigma$ in a finite number of points, i.e. for every $\eps>0$ sufficiently small $\eps p\in \Q^{+}\backslash \Sigma$. This concludes the proof.\end{proof}
\vspace{0.3cm}
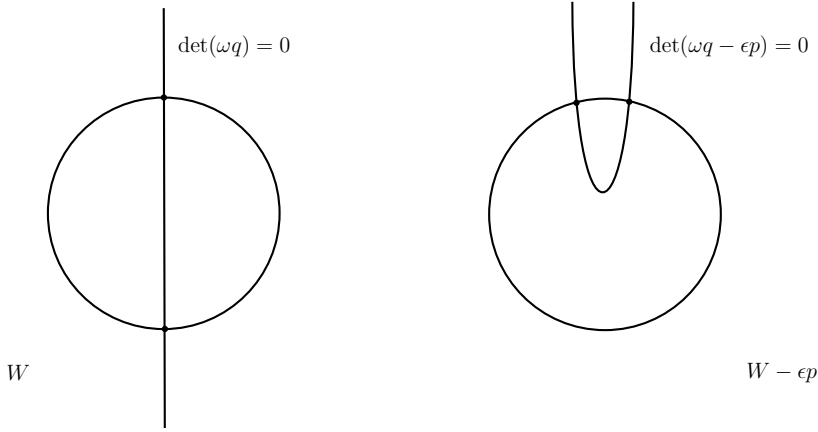
\begin{center}
\fontsize{12}{0}
\scalebox{0.7} 
{
\begin{pspicture}(0,-4.07)(15.72291,4.07)
\pscircle[linewidth=0.04,dimen=outer,fillstyle=solid](3.2610157,0.03){2.22}
\pscircle[linewidth=0.04,dimen=outer,fillstyle=solid](11.641016,0.01){2.22}
\psline[linewidth=0.04cm](3.2610157,3.93)(3.2810156,-4.05)
\psbezier[linewidth=0.04](11.021015,4.05)(11.038329,-0.83)(12.163702,-0.7251527)(12.181016,4.05)
\psdots[dotsize=0.12](3.2610157,2.23)
\psdots[dotsize=0.12](3.2810156,-2.17)
\psdots[dotsize=0.12](11.101016,2.13)
\psdots[dotsize=0.12](12.101016,2.15)
\usefont{T1}{ptm}{m}{n}
\rput(0.5,-3.0){$W$}
\usefont{T1}{ptm}{m}{n}
\rput(4.6,3.2){$\det(\omega q)=0$}
\usefont{T1}{ptm}{m}{n}
\rput(15,-3.0){$W-\eps p$}
\usefont{T1}{ptm}{m}{n}
\rput(14,3.2){$\det( \omega q-\eps p)=0$}
\end{pspicture} 
}
\end{center}
\vspace{0.3cm}
Since the codimension of $Z^{(r)}$ is $\binom{r+1}{2}$ we can immediately derive the following.
\begin{coro}\label{codim}Assume $r>\frac{1}{2}(-1+\sqrt{8k-7})$; then there exists a positive definite form $p$ such that for $\eps>0$ small enough:
$$\{\omega \in S^{k-1}\,|\,\dim \ker (\omega q-\eps p)\geq r\}=\emptyset$$
\end{coro}
We recall the following  result describing the local topology of the space of quadratic forms (see \cite{AgLe}, Proposition 9).
\begin{propo}\label{topquad}Let $q_{0}\in \mathcal{Q}$ be a quadratic form and let $V$ be its kernel. Then there exists a neighborhood $U_{q_{0}}$ of $q_{0}$ and a smooth semialgebraic map $\phi:U_{q_{0}}\to \Q(V)$ such that: 1) $\phi(q_{0})=0$; 2) $\ii^{-}(q)=\ii^{-}(q_{0})+\ii^{-}(\phi(q));$ 3) $\dim\ker (q)=\dim\ker(\phi(q));$ 4) for every $p\in \mathcal{Q}$ we have $ d\phi_{q_{0}}(p)=p|_{V}.$
\end{propo}

Combining Lemma \ref{perturb} and the previous proposition we can prove the following corollary, which shows that after the perturbation the sets $\Omega_{n-j}(\eps)$ have a very nice structure.
\begin{coro}\label{boundary}Let $p$ be the positive definite form provided by Lemma \ref{perturb}. Then for every $\eps>0$ small enough:
$$\Omega_{n-j}(\eps)\textrm{ is a \emph{topological} submanifold of $S^{k-1}$ with boundary}.$$
\end{coro}
\begin{proof}
Let $p$ be the quadratic form given by Lemma \ref{perturb} and $f=q_\eps:S^{k-1}\to \Q_{n+1}$ the map consequently defined. Let us consider a point $\omega$ in $\Omega_{n-j}(\eps)$ and the map $\phi:U_{f(\omega)}\to \Q(\ker f(\omega))$ given by Proposition \ref{topquad}. Since $d\phi_{f(\omega)}p=p|_{\ker f(\omega)}$ then $d\phi_{f(\omega)}$ is surjective. On the other hand by transversality of $f$ to each stratum $N_{r}$ we have:
$$ \textrm{im} (df_{\omega})+T_{f(\omega)}N_{r}=\mathcal{Q}_{n+1}$$
Since $\phi(N_{r})=\{0\}$ (notice that this condition implies $(d\phi_{f(\omega)})|_{T_{f(\omega)}N_{r}}=0$) then
$$\Q(\ker f(\omega))=\textrm{im}(d\phi_{f(\omega)})=\textrm{im}(d(\phi\circ f)_{\omega})$$
which tells that $\phi \circ f$ is a submersion at $\omega.$ Thus by the Rank Theorem there exist an open neighborhood $U_{\omega}$ of $\omega$ and an open diffeomorphism onto its image $\psi$ such that the following diagram is commutative:
$$\begin{tikzpicture}[xscale=1.5, yscale=1.5]

    \node (A2_0) at (2, 0) {$\Q(\ker f(\omega))$};
    \node (A1_1) at (1, 1) {$U_\omega$};
    \node (A3_1) at (3, 1) {$\Q(\ker f(\omega))\times \R^{l}$};
    
    \path (A1_1) edge [->] node [auto] {$\psi$} (A3_1);
    \path (A1_1) edge [->] node [auto,swap] {$\phi \circ f$} (A2_0);
    \path (A3_1) edge [->] node [auto] {$p_{1}$} (A2_0);

      \end{tikzpicture}
$$ 
(in particular $\psi(U_{\omega})$ is an open subset of $\Q(\ker f(\omega))\times \R^{l}$). Let us pick up an open neighborhood of $\psi(\omega)$ of the form $A\times B$, with $A\subset  \Q(\ker f(\omega))$ and $B\subset \R^{l}$ contractible, and consider the open set $U'=U_\omega\cap \psi^{-1}(A\times B)$ and the commutative diagram:
$$\begin{tikzpicture}[xscale=1.5, yscale=1.5]

    \node (A2_0) at (2, 0) {$A$};
    \node (A1_1) at (1, 1) {$U'$};
    \node (A3_1) at (3, 1) {$A\times B$};
    
    \path (A1_1) edge [->] node [auto] {$\psi$} (A3_1);
    \path (A1_1) edge [->] node [auto,swap] {$\phi \circ f$} (A2_0);
    \path (A3_1) edge [->] node [auto] {$p_{1}$} (A2_0);

      \end{tikzpicture}
$$ 

Notice now that for every $\eta$ in $U'$ the second point of Proposition \ref{topquad} implies that $\ii^{-}(f(\eta))=\ii^{-}(f(\omega))+\ii^{-}(\phi(f(\eta))).$ In particular we see that $U'\cap \Omega_{n-j}(\eps)$ is homeomorphic, through $\psi$, to the set$$(A\cap \{q\in \Q(\textrm{ker} f(\omega))\textrm{ such that }\ii^-(q)\leq n-j-\ii^-(f(\omega))\} )\times B$$
The left hand side factor is the intersection of $A$ with the set of quadratic forms in $\Q(\textrm{ker} f(\omega))$ with \emph{negative} inertia index less or equal than $n-j-\ii^- (f(\omega));$ since this set is a topological submanifold with boundary in $\Q(\textrm{ker} f(\omega))$, then $U'\cap \Omega_{n-j}$ is homeomorphic to an open neighborhood of a topological manifold with boundary. This proves that for every point $\omega\in \Omega_{n-j}(\eps)$ there is an open neighborhood $U'_{\omega}$ such that $U'_{\omega}\cap \Omega_{n-j}(\eps)$ is homeomorphic to an open set of a topological manifold with boundary; thus $\Omega_{n-j}(\eps)$ itself is a topological manifold with boundary (the boundary being possibly empty). 
\end{proof}
\section{A topological bound}
The scope of this section is to provide a formula which generalizes (\ref{toyeq}) from Example \ref{toy}. The idea is to use Lemma \ref{union} and Corollary \ref{boundary} in Agrachev's bound: the first says that we can perturb each set $\Omega^j$ to a set $\Omega_{n-j}(\eps)$ without changing its Betti numbers, the second says that we can do that \emph{and} make the new sets topological manifolds with boundary. As we will see we can use the topological manifold structure of these sets to get more information out of Agrachev's bound. We start by proving the following lemma from algebraic topology.
\begin{lemma}\label{boundary2} Let $M$ be a semialgebraic topological submanifold of the sphere $S^n$ with nonempty boundary and nonempty interior. Then:
$$b(M)=\frac{1}{2}b(\partial M)$$
\end{lemma}
\begin{proof}
By assumption also $N=S^n\backslash \textrm{int}(M)$ is a semialgebraic topological manifold with boundary $\partial N=\partial M$.  Let us consider collar neighborhoods $A$ of $M$ and $B$ of $N$ such that $A\cap B$ deformation retracts to $\partial M$ (such collar neighborhoods certainly exist by semialgebraicity and the Collaring Theorem). From the reduced Mayer-Vietoris sequence for the pair $(A, B)$ we get: $\tilde b_i(A)+ \tilde b_i(B)=\tilde b_i(A\cap B)$ for $i\neq n-1$, and $\tilde b_{n-1}(A)+ \tilde b_{n-1}(B)= \tilde b_{n-1}(A\cap B)-1$. Summing up all these equations we obtain:
$$\tilde b(A)+ \tilde b(B)=\tilde b(A\cap B) -1$$
(here we are using the notation $\tilde b(Y)$ for the sum of the \emph{reduced} Betti numbers of a semialgebraic set $Y$).
Alexander-Pontryiagin duality implies that $\tilde b(N)=\tilde b(M)$; on the other hand $A$ and $B$ deformation retract respectively to $M$ and $N$, which means $\tilde b(A)=\tilde b(M)=\tilde b(N)=\tilde b(B)$. Plugging this equality in the previous formula immediately gives the statement.
\end{proof}
We prove now the main technical theorem of the paper. A toy model proof in the case $\Sigma_\eps$ is smooth was provided in Example \ref{toy}; another proof for the case $\Sigma_\eps$ has only isolated singularities is given in Example \ref{ex:four}; the reader uncomfortable with technical details is advised to take a look at them first.
\begin{teo}\label{main}Let $X$ be defined by the vanishing of the quadratic forms $q_1,\ldots, q_k$ in $\RP^n.$ For every quadratic form $p$ and real number $\eps$ let us define 
$$\Sigma_\eps=\{\omega \in S^{k-1}\,|\, \det (\omega q-\eps p)=0\}$$ 
where $q=(q_1,\ldots, q_k)$. There exists a positive definite form $p\in Q_{n+1}$ such that for every $\eps>0$ small enough:
\begin{itemize}
\item[i)] the map $q_\eps:S^{k-1}\to \Q_{n+1}$ given by $\omega \mapsto \omega q-\eps p$ is transversal to all strata of $Z$, stratified as  in equation (\ref{strataZ}); in particular 
$$\Sigma_\eps^{(r)}=\{\omega \in S^{k-1}\,|\, \dim \ker (\omega q-\eps p)\geq r\}$$
is an algebraic subset of $S^{k-1}$ of codimension $\binom{r+1}{2}$.
\item[ii)] if we let $\mu$ and $\nu$ be respectively the maximum and the minimum of the negative inertia index on the image of $q_\eps$, then
\begin{equation}\label{boundstrong}b(X)\leq n+1-2(\mu-\nu)+\frac{1}{2}\sum_{r\geq1} b\big(\Sigma_\eps^{(r)}\big)\end{equation}
\end{itemize}
The last sum is indeed finite since for $\binom{r+1}{2}\geq k$ part \emph{i)} implies $\Sigma_\eps^{(r)}=\emptyset$.

\end{teo}

\begin{proof}
The first part of the statement follows directly from Lemma \ref{perturb}. For the second part, in order to get the $-2(\mu-\nu)$ term in equation (\ref{boundstrong}), we will use a refined version of Agrachev's bound. The refined bound follows by considering in the proof of Agrachev's one a spectral sequence converging directly to $H_*(X)$ and whose second term is isomoprhic to $E_2\simeq\oplus_j H^*(B^{k}, \Omega^{j+1})$, where $B^{k}$ is the unit ball in $\R^{k}$ such that $\partial B^{k}=S^{k-1}.$  The existence of such a spectral sequence is the content of Theorem A of \cite {AgLe}; repeating verbatim the above argument we get $b(X)\leq \sum_{j\geq 0} b(B, \Omega^{j+1})$. Let now $p$ be given by Lemma \ref{perturb}; since $p$ is positive definite then Lemma \ref{union} implies $b(\Omega^{j+1})=b(\Omega_{n-j}(\eps))$ for $\eps>0$ sufficiently small and for every $j\geq 0.$ In particular we can rewrite the refined Agrachev's bound as:
\begin{equation}\label{refined}b(X)\leq n+1-2(\mu-\nu)+\sum_{\nu\leq j\leq \mu-1}b(\Omega_j(\eps))\end{equation}
The rest of the proof is devoted to bound the last term $\sum b(\Omega_j(\eps)).$ First notice that Corollary \ref{boundary} says that each nonempty $\Omega_{j}(\eps)$ is a topological submanifold of $S^{k-1}$ with boundary and nonempty interior; thus applying Lemma \ref{boundary2} we get for such a manifold $$b(\Omega_j(\eps))=\frac{1}{2}b(\partial \Omega_j(\eps)),\quad \nu\leq j\leq \mu-1$$ For convenience of notation let us rename these boundaries as following: 
$$C_j=\partial \Omega_{\nu+j-1}(\eps),\quad j=1, \ldots, l=\mu-\nu$$ Thus (\ref{refined}) can be rewritten as:
\begin{equation}\label{refined2}b(X)\leq n+1-2(\mu-\nu)+\frac{1}{2}\sum_{j=1}^{l}b(C_j).\end{equation}
Let us now analize the structure of $\Sigma_\eps=\Sigma_\eps^{(1)}.$ By construction this set equals the union of all the $C_j$'s, but this union is not disjoint since $\Sigma_\eps$ might have singularities and these singularities precisely occur when two sets $C_j$ and $C_{j+1}$ intersect (this immediately follows from the fact that $q_\eps$ is transversal to all the strata of $Z$ and that $\Sigma_\eps$ is stratified by the preimages of the strata of $Z$ as described in Lemma \ref{perturb}). For convenience of notation, let us write $S(\omega, j)$ if the point $\omega$ has the property that there exists a sequence $\{\omega_n\}_{n\geq0}$ converging to $\omega$ such that $\ii^-(q_\eps(\omega_n))\geq j$. Corollary \ref{boundary} implies now that $C_j=\Omega_{\nu+j-1}(\eps)\cap \textrm{Cl}(\Omega_{\nu+j-1}(\eps)^c)$, i.e.
\begin{equation}\label{cj}C_j=\{\omega \,|\, \ii^-(q_\eps(\omega))\leq j\quad\textrm{and} \quad S(\omega, j+1)\}\end{equation}
Let us denote by $I_r$ the set of all subsets $\alpha$ of $\{1,\ldots, l\}$ consisting of $r$ consecutive integers; if $\alpha=\{\alpha_1,\ldots, \alpha_r\}\in I_r$ let us assume its elements are ordered in increasing way $\alpha_1\leq\cdots \leq \alpha_r$. Let now $r\in\{1,\ldots, l\}$, $\alpha \in I_r$ and for $i\in\{1,\ldots, l-r\}$ consider the sets:
$$E_{i,r}=\bigcup_{\alpha_1\leq i}\bigcap_{j\in\alpha}C_j,\quad F_{i+1,r}=\bigcap_{j=i+1}^{i+r}C_j$$
For example if $r=1$ we have $E_{i,1}=C_1\cup\cdots \cup C_i$ and $F_{i+1,1}=C_{i+1}$; if $r=2$ then $E_{i,2}=(C_{1}\cap C_2)\cup\cdots\cup (C_{i-1}\cap C_{i})$ and $F_{i+1,2}=C_{i+1}\cap C_{i+2}.$ We have the following combinatorial properties: $$E_{i,r}\cup F_{i+1,r} =E_{i+1,r}\quad \textrm{and}\quad E_{i,r}\cap F_{i+1,r}=\bigcap_{j=i}^{i+r}C_j.$$
The first equality is clear from the definition; for the second one notice that equation (\ref{cj}) implies: \begin{align*} E_{i,r}\cap F_{i+1,r}&=\bigcup_{\alpha_1\leq i}\{\omega\,|\,\ii^-(q_\eps(\omega))\leq \alpha_1\quad \textrm{and}\quad S(\omega, l+r+1)\}\\
&= \{\omega\,|\,\ii^-(q_\eps(\omega))\leq i\quad \textrm{and}\quad S(\omega, l+r+1)\}\\
&=\bigcap_{j=i}^{i+r}C_j.\end{align*} 
Plugging these equalities in the semialgebraic Mayer-Vietoris exact sequence of the pair $(E_{i,r}, F_{i+1,r})$ we get: \begin{equation}\label{sum}b\bigg( \bigcap_{j=i+1}^{i+r}C_j\bigg)\leq b\bigg(\bigcup_{\alpha_1\leq j+1}\bigcap_{j\in\alpha}C_j \bigg)+b\bigg(\bigcap_{j=i}^{i+r}C_j\bigg)-b\bigg(\bigcup_{\alpha_1\leq j}\bigcap_{j\in\alpha}C_j\bigg)\end{equation}
If we take now the sum of all these inequalities we obtain:
\begin{equation}\label{eq}\sum_{i=0}^{l-r} b\bigg( \bigcap_{j=i+1}^{i+r}C_j\bigg)\leq b\bigg(\bigcup_{\alpha_1\leq l-r+1}\bigcap_{j\in\alpha}C_j \bigg)+\sum_{i=0}^{l-r-1} b\bigg(\bigcap_{j=i}^{i+r}C_j\bigg) \end{equation}
In fact when we take the sum of all inequalities (\ref{sum}) all the first and the last terms in the r.h.s. cancel (since they appear with opposite signs), except for the last inequality which gives the contribution $b\big(\bigcup_{\alpha_1\leq l-r+1}\bigcap_{j\in\alpha}C_j \big).$ Moreover since $q_\eps$ is transversal to all strata of $Z$, then Proposition \ref{zetar} implies:
$$\bigcup_{\alpha_1\leq l-r+1}\bigcap_{j\in\alpha}C_j =\bigcup_{\alpha\in I_r}\bigcap_{j\in \alpha}C_j=\Sigma_\eps^{(r)}$$
Substituting this formula into equation (\ref{eq}) we finally get:
\begin{equation}\label{interm}\sum_{i=0}^{l-r} b\bigg( \bigcap_{j=i+1}^{i+r}C_j\bigg)\leq b\big(\Sigma_\eps^{(r)}\big)+\sum_{i=0}^{l-r-1} b\bigg(\bigcap_{j=i}^{i+r}C_j\bigg) \end{equation}
In particular we have the following chain of inequalities (we keep on substituting at each step what we get from (\ref{interm})):
\begin{align*}\sum_{i=1}^{l}b(C_i)&=\sum_{i=0}^{l-1} b\bigg( \bigcap_{j=i+1}^{i+1}C_j\bigg)\leq b\big(\Sigma_\eps^{(1)}\big)+\sum_{i=0}^{l-2} b\bigg(\bigcap_{j=i}^{i+1}C_j\bigg)\\
&\leq b\big(\Sigma_\eps^{(1)}\big)+b\big(\Sigma_\eps^{(2)}\big)+\sum_{i=0}^{l-3} b\bigg(\bigcap_{j=i}^{i+2}C_j\bigg)\leq \cdots\\
&\leq \sum_{r\geq 1}b\big(\Sigma_\eps^{(r)}\big)\end{align*}
Substituting this into equation (\ref{refined2}) and recalling that $\Sigma_\eps^{(r)}$ is empty for $\binom{r+1}{2}\geq k$ the result follows.
\end{proof}

As a corollary we immediately get the following theorem.

\begin{teo}\label{topological}Let $\sigma_k=\lfloor \frac{1}{2}(-1+\sqrt{8k-7})\rfloor$; then 
$$\emph{\textrm{Topological bound:}}\quad  b(X)\leq b(\RP^n)+\frac{1}{2}\sum_{r=1}^{\sigma_k} b\big(\Sigma_\eps^{(r)}\big)$$
\end{teo}
\begin{proof}This is simply a reformulation of previous theorem in a nicer form. In fact $\mu-\nu\geq 0$, $n+1=b(\RP^n)$ and $\sigma_k$ is given by Corollary \ref{codim}.
\end{proof}

In the particular case we assume also nondegeneracy of the linear system $W$ we get the following theorem.

\begin{teo}\label{nondegeneracy}For a generic choice of $W=\textrm{\emph{span}}\{q_1, \ldots, q_k\}$ and $r\geq 1$
 $$\Sigma_W^{(r)}=\{q \in W\backslash\{0\}\,|\, \dim \ker (q)\geq r\}=\textrm{\emph{Sing}}\big(\Sigma_W^{(r-1)}\big)$$
 and the following formula holds:
$$b(X)\leq b(\RP^n)+\frac{1}{2}\sum_{r\geq 1}b\big(\Sigma_W^{(r)}\big)$$
\begin{proof}
Let us fix a scalar product on $\Q_{n+1}$; then for a generic choice of $q_1, \ldots, q_k$ the unit sphere $S^{k-1}$ in $W$ is transversal to all strata of $Z=\coprod N_r$ and the first part of the statement follows from equation (\ref{sing}).\\
Notice that the set of linear affine embeddings $f:\R^k\to \Q_{n+1}$ whose restriction to $S^{k-1}$ is transversal to all the strata of $Z$ is an open dense set; moreover if two such embeddings $f_0$ and $f_1$ are joined by a nondegenerate homotopy, then by Thom Isotopy Lemma the two sets $f_0^{-1}(Z^{(r)})$ and $f_1^{-1}(Z^{(r)})$ are homotopy equivalent. In particular for $\eps>0$ small enough the map $q_\eps$ given by Theorem \ref{main} is nondegenerate homotopic to $S^{k-1}\hookrightarrow \Q_{n+1}$ and thus for every $r\geq 0$ we can substitute:
$$b\big( \Sigma_\eps^{(r)}\big)=b\big(\Sigma_W^{(r)}\big)$$
in equation (\ref{boundstrong}), which gives the result.
\end{proof}
\end{teo}
\begin{remark}
From the point of view of classical agebraic geometry, it is natural to consider the projectivization $\mathbb{P}W$ rather than $W$ itself; similarly we can consider $\mathbb{P}\Sigma$ and by the Gysin exact sequence we get, for a generic $W$:
$$b(X)\leq b(\RP^n)+\sum_{r\geq 1}b\big( \mathbb{P}\Sigma_W^{(r)}\big)$$
Unfortunately there is no such formula for the general case: this is due to the fact that the perturbation $\Sigma_\eps$ is not invariant by the antipodal map.\\
Moreover, the way we got our formula shows that if we are interested in the topology of the complement of $X$, then we can remove the $b(\RP^n)$ from the sum and get again for a generic choice of $W$:
$$b(\RP^n\backslash X)\leq\sum_{r\geq 1} b\big( \mathbb{P}\Sigma_W^{(r)}\big)$$
\end{remark}
\section{A numerical bound}

From the previous discussion we can derive quantitative bounds on the homological complexity of the intersection of real quadrics. We start by proving the following proposition, which essentially refines Corollary 2.3 of \cite{BaBa}. 
\begin{propo}\label{sphericalbound}
Let $Y\subset S^{k-1}$ be defined by polynomial equations of degree less or equal than $d$. Then:
$$b(Y)\leq (2d)^{k-1}+\frac{1}{8}\binom{k+1}{3}(6d)^{k-2}$$
\end{propo}
\begin{proof}
Using Alexander-Pontryiagin duality our problem is  equivalent to that of bounding $b(S^{k-1}\backslash Y)=b(Y).$  Let $Y$ be defined on the sphere by the polynomials $f_1, \ldots, f_R$ and consider the new polynomial $F=f_1^2+\cdots+f_R^2$; then clearly $Y$ is defined also by $\{F=0\}$ on the sphere, and since $F\geq 0$ we have $S^{k-1}\backslash Y=\{F|_{S^{k-1}}>0\};$ notice that the degree of $F$ is $\delta=2d.$ By semialgebraic triviality for $\eps>0$ small enough we have the homotopy equivalences:
$$S^{k-1}\backslash Y\sim\{F|_{S^{k-1}}>\eps\}\sim \{F|_{S^{k-1}}\geq \eps\}$$
Let now $\eps>0$ be a small enough regular value of $F|_{S^{k-1}}$; then $\{F|_{S^{k-1}}\geq \eps\}$ is a submanifold of the sphere with smooth boundary $\{F|_{S^{k-1}}=\eps\}$ and by Lemma \ref{boundary2} we obtain:
$$b\big(S^{k-1}\backslash Y\big)=\frac{1}{2}b(\{F|_{S^{k-1}}=\eps\}).$$
Thus we reduced to study the topology of $\{F|_{S^{k-1}}=\eps\}$: this set is given in $\R^k$ by the two equations $F-\eps=0$ and $\|\omega\|^2-1=0$. Equivalently we can consider their homogenization $g_1={}^hF-\eps \omega_0^{2d}=0$ and $g_2=\|\omega\|^2-\omega_0^2=0$ and their common zero locus in $\RP^k$: since there are no common solutions on $\{\omega_0=0\}$ (the hyperplane at infinity) these two equations still define $\{F|_{S^{k-1}}=\eps\}$. By Fact 1 in \cite{Le3} it follows that we can \emph{real} perturb the coefficients of $g_1$ and $g_2$ and make their common zero set in $\CP^k$ a smooth complete intersection. This perturbation of the coefficients will not change the topology of the zero locus set in $\RP^k$ since before the perturbation it was a smooth manifold; the fact that the perturbation is \emph{real} allows us to use Smith's theory. Thus let $\tilde g_1$ and $\tilde g_2$ be the perturbed polynomials; we have:
$$b(\{F|_{S^{k-1}}=\eps\})=b(Z_{\RP^k}(\tilde g_1, \tilde g_2))\leq b(Z_{\CP^k}(\tilde g_1, \tilde g_2))$$
where in the last step we have used Smith's inequalities. Eventually we end up with the problem of bounding the homological complexity of the complete intersection $C$ of multidegree $(2,\delta)$ in $\CP^k.$ Let us compute first the Euler characteristic of $C.$ By Hirzebruch's formula this is given by the $(k-2)$th coefficient in the series expansion around zero of the function:
$$H(x)=\frac{2\delta(1+x)^{k+1}}{(1+2x)(1+\delta x)}$$
In other words we have:
$$\chi(C)=\frac{H^{(k-2)}(0)}{{(k-2)!}}$$
To compute this number let us write $H(x)= F(x)G(x)$ with $F(x)=\frac{2\delta(1+x)^{k+1}}{1+2x}$ and $G(x)=\frac{1}{1+\delta x}.$ In this way we have:
$$H^{(k-2)}(0)=\sum_{j=0}^{k-2}\binom{k-2}{j}F^{(j)}(0)G^{(k-2-j)}(0).$$
To compute the derivatives of $F$ we do the same trick as for $H:$ we write $F(x)=A(x)B(x)$ where $A(x)=2\delta(1+x)^{k+1}$ and $B(x)=\frac{1}{1+2x}$. In this way, using the series expansion $B(x)=\sum_{i=0}^{\infty}(-1)^{i}2^{i}x^i$, we get:
$$F^{(j)}(0)=j!2(-2)^{j}\delta \sum_{i=0}^{j}\binom{k+1}{i}\bigg(-\frac{1}{2}\bigg)^i$$
We also have $G^{(i)}(0)=(-1)^{i}\delta^{i}i!$ (this last equality follows from the series expansion around zero $G(x)=\sum_{i=0}^{\infty}(-1)^{i}\delta^{i}x^i$). Plugging these equalities into the above one we get:
$$\chi(C)=(-1)^k\sum_{j=0}^{k-2}\bigg(\sum_{i=0}^j \binom{k+1}{i}\bigg(-\frac{1}{2}\bigg)^i\bigg)2^{j+1}\delta^{k-j-1}$$
Recall now that the formula $b(C)=(k-1)(1+(-1)^{k+1})+(-1)^k \chi(C)$ gives:
\begin{align*}b(C)&=(k-1)(1+(-1)^{k+1})+\sum_{j=0}^{k-2}\bigg(\sum_{i=0}^j \binom{k+1}{i}\bigg(-\frac{1}{2}\bigg)^i\bigg)2^{j+1}\delta^{k-j-1}\\
&= (k-1)(1+(-1)^{k+1})+2\delta^{k-1}+\sum_{j=1}^{k-2}\bigg(\sum_{i=0}^j \binom{k+1}{i}\bigg(-\frac{1}{2}\bigg)^i\bigg)2^{j+1}\delta^{k-j-1}\end{align*}
Since $|\sum_{i=0}^j \binom{k+1}{i}(-\frac{1}{2})^i|\leq \binom{k+1}{3}(\frac{3}{2})^{k-2}$, from the above equality we can deduce:
$$b(C)\leq   2(k-1)+2\delta^{k-1}+2\delta^{k-1}\bigg(\frac{3}{2}\bigg)^{k-2}\binom{k+1}{3}\sum_{j=1}^{k-2}\bigg(\frac{2}{\delta}\bigg)^j$$
Since now $\delta^{k-1}\sum_{j=1}^{k-2}(\frac{2}{\delta})^j=2(\delta^{k-2}+2\delta^{k-3}+\cdots+2^{k-3}\delta+2^{k-2})$ and $2^{k-2}\geq (k-1)$, we can bound $2(k-1)+ 2\delta^{k-1}\sum_{j=1}^{k-2}(\frac{2}{\delta})^j$ with $2^{k}\delta^{k-2}$ and finally write:
\begin{equation}\label{completedelta}b(C)\leq 2\delta^{k-1}+\frac{1}{4}\binom{k+1}{3}(3\delta)^{k-2}\end{equation}
The previous inequality, together with $b(Y)\leq \frac{1}{2}b(C)$ and $\delta=2d$ gives the result.\end{proof}

\begin{remark}
We notice that as long as $d$ is large enough with respect to $k$, the previous bound improves Milnor's one, which gives $b(Y)\leq d(2d-1)^{k-1}$; here it is essential that $Y$ is on the sphere, as we used Alexander-Pontryiagin duality.
\end{remark}
As a corollary we get the following theorem.

\begin{teo}\label{numerical} Let $X$ be the intersection of $k$ quadrics in $\RP^n$. Then:
$$b(X)\leq O(n)^{k-1}$$
\end{teo}

\begin{proof}
We use the bound given in Theorem \ref{topological}: the proof is essentially collecting the estimates given by the previous proposition for each summand $b\big(\Sigma_{\eps}^{(r)}\big)$. By construction we have that $\Sigma_\eps^{(r)}$ is a determinantal variety and it is defined by polynomials $f_1,\ldots, f_R$ of degree $d=n-r+2$ on the sphere $S^{k-1}$. The result now follows by plugging the bounds given in Proposition \ref{sphericalbound}  in the summands of Theorem \ref{topological} (there are only $\sigma_k$ of such summand). \end{proof}
\begin{remark}
As suggested to the author by S. Basu, there are two other possible ways to get such numerical estimates. The first one is using a general position argument similar to \cite{BaBa} and combinatorial Mayer-Vietoris bounds as in \cite{BPR}; the second one is using again a general position argument and stratified Morse theory (which in the semialgebraic case is very well controlled, as noticed in \cite{BasuCell}). Both this approaches work also in the \emph{affine} case producing a bound of the same shape; in the projective case it seems that also the leading \emph{coefficient} is the same. As for numerical uniform bounds, the advantage of the first one is that it is applicable to more general cases, i.e. besides the quadratic one. To the author knowledge nothing has been published on the subject: together with S. Basu he plans to give an account of these different techniques in a forthcoming paper. \end{remark}
We introduce the following notation:
$$B(k,n)=\max\{b(X)\,|\,\textrm{$X$ is the intersection of $k$ quadrics in $\RP^n$}\}.$$
We discuss now the sharpness of the previous bound, showing that:
$$B(k,n)=O(n)^{k-1}$$
Theorem \ref{numerical} gives the inequality $B(k,n)\leq O(n)^{k-1}$; for the opposite inequality we need to produce for every $k$ and $n$ an intersection $M_\R$ of $k$ quadrics in $\RP^n$ with $b(M_\R)=O(n)^{k-1}.$
Let us first notice that repeating the same argument of Proposition \ref{sphericalbound}, we can deduce that the complete intersection $M$ of $k$ quadrics in $\CP^n$ has:
$$b(M)=b(M;\Z)=O(n)^{k-1}$$
(this computation is already performed in \cite{BaBa}). It is a known result that there exists a real maximal $M$, i.e. a complete intersection of $k$ \emph{real} quadrics in $\CP^n$ whose real part $M_\R$ satisfies:
$$b(M_\R)=b(M)$$
Such an existence result holds in general for any complete intersection of multidegree $(d_1, \ldots, d_k)$.  An asymptotic construction is provided in \cite{ItVi1}; the proof for the general case has not been published yet but the author has been informed that it will be the subject of a forthcoming paper of the same authors as \cite{ItVi1}. 
\section{Examples}
\begin{example}[$k=2$]
In the case $X$ is the intersection of \emph{two} quadrics in $\RP^n$, the ideas previously discussed produces the sharp bound $b(X)\leq 2n$: in fact by inequality (\ref{refined2}) we have:
$$b(X)\leq n+1-2(\mu-\nu) +\frac{1}{2}b(\Sigma_\eps)$$ 
(every $\Sigma_\eps^{(r)}$ with $r>1$ is empty). On the other hand $\Sigma_\eps$ is defined by an equation of degree $n+1$ on the circle $S^1$ and thus it consists of at most $2(n+1)$ points. This gives:
$$b(X)\leq n+1-2(\mu-\nu)+n+1\leq 2n$$
(in the case $\mu=\nu$ we have $b(X)\leq n+1$ ). Moreover for every $n$ there exist two quadrics in $\RP^n$ whose intersection $X$ satisfies $b(X)=2n$ (see Example 2 of \cite{Le2}). Notice that the example provided there in the case $n$ \emph{odd} gives a \emph{singular} $X.$ Using the notation introduced above, this reads:
$$B(2,n)=2n$$
What is interesting now is that for \emph{odd} $n$  this number $B(2,n)$ is attained only by a singular intersection of quadrics: the nonsingular one has at most $b(X)\leq 2n-2$ (this follows from Smith's inequality and the Hirzebruch's formula for the complete intersection of two quadrics in $\CP^n$). For a more detailed discussion the reader is referred to section 6 of \cite{Le2}.
\end{example}

\begin{example}[$k=3$]
In the case $X$ is the intersection of three quadrics then, inequality (\ref{refined2}) gives:
$$b(X)\leq n+1-2(\mu-\nu) +\frac{1}{2}b(\Sigma_\eps)$$ 
Again, since the codimension of $\Sigma_\eps^{(r)}$ is greater than three for $r\geq 2$, then in this case all these sets except $\Sigma_\eps$ are empty (since $k=3$ these sets are subsets of the sphere $S^2$). This also says that $\Sigma_\eps$ is a smooth curve on $S^2$; let $f=\|\omega\|^2-1 $ and $g=\det(\omega q-\eps p)$  be the polynomials defining this curve and $F,G$ their homogenization. Then there exists a \emph{real} perturbation $\tilde G$ of $G$ that makes the common zero locus $C$ of $\tilde G$ and $F$ a smooth complete intersection in $\CP^3$. Since $\Sigma_\eps$ is smooth, the real part $C_\R$ of this complete intersection has the same topology of it and by Theorem \label{topological}:
$$b(X)\leq \frac{1}{2}b(C)+O(n)$$
Recall that equation (\ref{completedelta}) was for a complete intersection of multidegree $(2,\delta)$ in $\CP^3$; specifying it to this case $\delta=n+1$ we get $b(C)\leq 2n^2+O(n)$, which plugged into the previous inequality gives:
$$b(X)\leq n^2+O(n).$$ 
(indeed Theorem 1 of \cite{Le3} gives the refined bound $b(X)\leq n^2+n$).
We notice now that in this case:
$$B(3,n)=n^2+O(n).$$
In fact the previous inequality provides the upper bound, and the lower bound is given by the existence of almost maximal real complete intersections of three quadrics (see \cite{Krasnov2} for an explicit construction of such maximal varieties).\\
In the case $X$ is smooth, using the spectral sequence approach the authors of \cite{DeItKh} have proved that the maximum value $B_0(3,n)$ of $b_0(X)$ satisfies:
$$\frac{1}{4}(n-1)(n+5)-2<B_0(3,n)\leq \frac{3}{2}l(l-1)+2$$
where $l=\lfloor \frac{1}{2}n\rfloor+1$. Notice in particular that $\frac{1}{4}\leq\liminf \frac{B_0(n,3)}{n^2}\leq \frac{3}{8}$ as $n$ goes to infinity. \end{example}

\begin{example}[$k=4$]\label{ex:four} This is the first case where we need to take into account the complexity of the singular points of $\Sigma_\eps$. As promised, we give a simplified proof of part ii) of Theorem \ref{main} for this case, with the scope of getting the reader acquainted to the idea of that proof. Let $p\in \Q_{n+1}$ be the positive definite form given by Lemma \ref{perturb}, then by Lemma \ref{union} and Agrachev's bound we get:
$$b(X)\leq  \sum_{\nu\leq j\leq \mu-1}b( \Omega_j(\eps))+O(n)$$
where the $\Omega_{j}(\eps)$ are now subsets of the sphere $S^3.$
Corollary 9 says that each of them is a manifold with boundary; let us rename as above such boundaries: $C_j=\partial \Omega_{\nu+j-1}(\eps)$ for $j=1,\ldots, l=\mu-\nu.$ Lemma \ref{boundary2} allows now to use $\frac{1}{2}b(C_j)$ instead of $b(\Omega_{\nu+j-1}(\eps))$ in the previous bound, getting:
$$b(X)\leq \frac{1}{2}\sum_{j=1}^lb(C_j)+O(n)$$
Now we have that $\Sigma_\eps$ is a surface on $S^3$ given by $C_1\cup\cdots\cup C_l$, but this union is not disjoint since singular points may occur. These singular points are isolated, since their union (if nonempty) has codimension $3$ on the sphere $S^3$. The set $\Sigma_\eps^{(2)}=\textrm{Sing}(\Sigma_\eps)$ equals exactly the set of points where two different $C_j$ intersect. On the other hand if $|i-j|\geq2$ then $C_j\cap C_i=\emptyset$, since any point on this intersection would have kernel at least of dimension three. Thus $\Sigma_\eps$ is made up taking the abstract disjoint union of the sets $C_1, \ldots, C_l$ and identifying the points on $C_j\cap C_{j+1}$ for $j=1, \ldots, l-1.$ This identification procedure can increase the number of the generators of the fundamental group; the number of connected components instead can decrease of at most $b(\textrm{Sing}(\Sigma_\eps)),$ which is to say:
$$b(\Sigma_\eps)\geq b\bigg(\coprod_{j=1}^lC_j\bigg)-b(\textrm{Sing}(\Sigma_\eps))$$ 
The following picture shows an example how this identification procedure works: $\Sigma_\eps$ has two singular points and is obtained by glueing the disjoint union of $C_1$ and $C_2$ along two copies of these singular points (one copy is on $C_1$ and one on $C_2$).
\vspace{0.3cm}
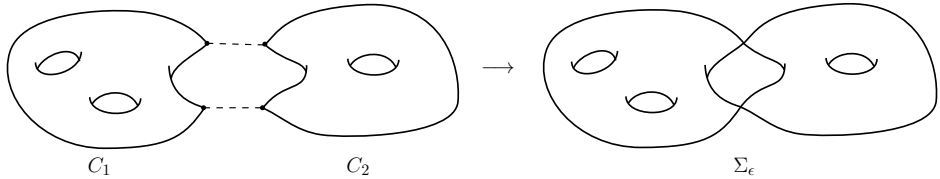
\begin{center}
\fontsize{15}{0}
\scalebox{0.55} 
{
\begin{pspicture}(0,-2.148847)(22.819946,2.1288471)
\psbezier[linewidth=0.04](5.1113057,1.0390087)(4.85164,1.3660076)(4.392474,1.6463987)(3.717313,1.7790587)(3.0421524,1.9117188)(0.68332964,1.899532)(0.34166482,0.84969366)(0.0,-0.20014468)(1.0796608,-1.4909295)(2.6239858,-1.4909295)(4.1683106,-1.4909295)(4.564642,-1.2155621)(5.0429726,-0.50993305)
\psbezier[linewidth=0.04](6.429306,-0.50993305)(6.866637,-0.68203765)(7.2133727,-1.0719202)(7.8642983,-1.1639307)(8.5152235,-1.2559413)(11.097085,-1.2170762)(11.168515,-0.36126745)(11.239946,0.49454126)(10.829948,2.0888472)(8.886195,1.9793557)(6.9424405,1.8698643)(6.770971,1.2283239)(6.494541,1.0155697)
\psbezier[linewidth=0.04](2.2686543,-0.2552181)(2.2686543,-0.7600582)(3.4989479,-0.7853002)(3.4989479,-0.2804601)
\psbezier[linewidth=0.04](2.326893,-0.4192911)(2.5311434,-0.08149369)(3.229179,-0.028040027)(3.4261494,-0.45715412)
\psbezier[linewidth=0.04](0.95666146,0.58631426)(0.95666146,0.05860586)(2.0536885,0.30267102)(2.0636556,0.8445869)
\psbezier[linewidth=0.04](0.9953732,0.41480902)(1.1311399,0.76790804)(1.8943169,1.0734297)(2.0237641,0.6157441)
\psbezier[linewidth=0.04](8.506628,0.674147)(8.506628,0.16930687)(9.736921,0.14406487)(9.736921,0.648905)
\psbezier[linewidth=0.04](8.564867,0.51007396)(8.769117,0.8478714)(9.467153,0.90132505)(9.664124,0.47221094)
\psbezier[linewidth=0.04](5.097639,1.0217983)(4.85164,0.79806226)(4.2776437,0.5226948)(4.236644,0.19569601)
\psbezier[linewidth=0.04](5.0429726,-0.50993305)(4.6466413,-0.3894598)(4.2093105,-0.18293421)(4.1819773,0.50548434)
\psbezier[linewidth=0.04](6.497639,1.0190088)(7.0169697,0.5715367)(7.3859677,0.64037853)(7.4953003,0.36501113)
\psbezier[linewidth=0.04](6.4429727,-0.52993304)(6.97597,0.17569602)(7.5226336,-0.16851328)(7.508967,0.5199053)
\psbezier[linewidth=0.04](18.071306,1.0590087)(17.81164,1.3860075)(17.352474,1.6663986)(16.677313,1.7990588)(16.002151,1.9317188)(13.64333,1.919532)(13.301664,0.86969364)(12.96,-0.18014467)(14.03966,-1.4709295)(15.583985,-1.4709295)(17.128311,-1.4709295)(17.524641,-1.1955621)(18.002974,-0.48993304)
\psbezier[linewidth=0.04](17.989305,-0.48993304)(18.426638,-0.6620377)(18.773373,-1.05192)(19.424297,-1.1439307)(20.075224,-1.2359413)(22.657085,-1.1970762)(22.728516,-0.34126747)(22.799946,0.51454127)(22.389948,2.1088471)(20.446196,1.9993557)(18.502441,1.8898643)(18.33097,1.2483239)(18.05454,1.0355697)
\psbezier[linewidth=0.04](15.228654,-0.23521808)(15.228654,-0.7400582)(16.458948,-0.7653002)(16.458948,-0.26046008)
\psbezier[linewidth=0.04](15.286893,-0.39929113)(15.491143,-0.061493687)(16.189178,-0.008040028)(16.386148,-0.43715414)
\psbezier[linewidth=0.04](13.916661,0.60631424)(13.916661,0.07860586)(15.013688,0.322671)(15.023656,0.8645869)
\psbezier[linewidth=0.04](13.955373,0.43480903)(14.09114,0.787908)(14.854317,1.0934297)(14.983764,0.63574415)
\psbezier[linewidth=0.04](20.066628,0.694147)(20.066628,0.18930689)(21.29692,0.16406487)(21.29692,0.668905)
\psbezier[linewidth=0.04](20.124866,0.53007394)(20.329117,0.8678714)(21.027153,0.921325)(21.224123,0.49221095)
\psbezier[linewidth=0.04](18.057638,1.0417984)(17.81164,0.81806225)(17.237644,0.5426948)(17.196644,0.215696)
\psbezier[linewidth=0.04](18.002974,-0.48993304)(17.606642,-0.36945978)(17.16931,-0.16293421)(17.141977,0.5254844)
\psbezier[linewidth=0.04](18.057638,1.0590087)(18.57697,0.6115367)(18.945967,0.68037856)(19.0553,0.40501112)
\psbezier[linewidth=0.04](18.002974,-0.48993304)(18.53597,0.215696)(19.082634,-0.12851328)(19.068966,0.5599053)
\psdots[dotsize=0.12](5.099947,1.0488482)
\psdots[dotsize=0.12](5.019947,-0.5111518)
\psdots[dotsize=0.12](6.439947,-0.5111518)
\psdots[dotsize=0.12](6.499947,1.0288482)
\psline[linewidth=0.027999999cm,linestyle=dashed,dash=0.16cm 0.16cm](5.099947,1.0488482)(6.479947,1.0088482)
\psline[linewidth=0.027999999cm,linestyle=dashed,dash=0.16cm 0.16cm](5.019947,-0.5111518)(6.419947,-0.5111518)
\usefont{T1}{ptm}{m}{n}
\rput(2.5014021,-1.9261518){$C_1$}
\usefont{T1}{ptm}{m}{n}
\rput(8.761402,-1.9261518){$C_2$}
\usefont{T1}{ptm}{m}{n}
\rput(18.091402,-1.9261518){$\Sigma_\eps$}
\usefont{T1}{ptm}{m}{n}
\rput(12.131402,0.41384822){$\longrightarrow$}
\end{pspicture} 
}
\end{center}
\vspace{0.3cm}
Plugging this into the above inequality for $b(X)$ we get:
$$b(X)\leq \frac{1}{2}\big\{b(\Sigma_\eps)+b(\textrm{Sing}(\Sigma_\eps))\big\}+O(n)$$
Proposition \ref{sphericalbound} says that both $\Sigma_\eps$ and $\textrm{Sing}(\Sigma_\eps)$ have complexity bounded by $16n^3$ and thus:
\begin{equation}\label{four}b(X)\leq 16n^3+O(n).\end{equation}
On the other hand we notice that if $X$ is the real part of a smooth complete intersection of four real quadrics in $\CP^n$, then $b(X)\leq \frac{2}{3}n^3+O(n)$; thus the above bound is sharp only in its shape. By the above discussion on the topology of determinantal varieties, we expect (\ref{four}) can be improved.\end{example}

\end{document}